\documentclass[a4paper,11pt,reqno]{amsart}

\usepackage{amssymb}
\usepackage[all]{xy}

\usepackage{hyperref}
\usepackage{xcolor}
\usepackage{enumitem}
\usepackage{subcaption}
\usepackage[final]{graphicx}
\usepackage{float}
\usepackage{epic}
\usepackage{setspace}

\usepackage[all]{xy}
\usepackage{color}

\usepackage{verbatim}

\usepackage{tikz}
\usetikzlibrary[topaths]

\newcount\mycount

\numberwithin{equation}{section}

\newtheorem{introtheorem}{Theorem}

\newtheorem{theorem}{Theorem}[section]

\newtheorem{lemma}[theorem]{Lemma}

\newtheorem{proposition}[theorem]{Proposition}

\newtheorem*{theorem*}{Theorem}
\newtheorem*{question*}{Question}
\newtheorem*{corollary*}{Corollary}

\theoremstyle{definition}
\newtheorem{definition}[theorem]{Definition}

\theoremstyle{remark}

\theoremstyle{remark}

\newtheorem{remark}[theorem]{Remark}
\newtheorem{question}[theorem]{Question}

\begin{document}

\title{Word-Length Spectral Triples of $(\mathbb{Z}/2\mathbb{Z})\wr\mathbb{F}_{d}$ Are Not Metric}

\author{Mario Klisse}

\address{CAU Kiel, Department of Mathematics,
Heinrich-Hecht-Platz 6, 24118 Kiel, Germany}

\email{klisse@math.uni-kiel.de}

\begin{abstract}

Given a countable discrete group equipped with a proper length function, one can construct a natural spectral triple on its reduced group C$^{\ast}$-algebra. A well-studied question in non-commutative metric geometry is whether the Connes pseudo-metric associated with such a triple recovers the weak$^{\ast}$-topology on the state space, thereby yielding a compact quantum metric space in the sense of Rieffel. While this metric property is known to hold for several classes of groups - including those of polynomial growth and word-hyperbolic groups - it was widely expected that not every word-length function induces a compact quantum metric space. Despite this, no explicit counterexample has been identified to date. In this note, we provide the first family of counterexamples by proving that for every integer $d \geq 2$ the canonical spectral triple of the Lamplighter group $(\mathbb{Z}/2\mathbb{Z})\wr\mathbb{F}_{d}$, equipped with the word-length function associated with a finite symmetric generating set, fails to be a spectral metric space.

\end{abstract}

\date{\today. \emph{MSC2020:} 46L87, 58B34, 46L89, and 20F65.}

\maketitle

\section*{Introduction}

\vspace{3mm}

A central objective of non-commutative geometry is the translation of classical geometric data into an operator-algebraic framework. The foundational building block for this dictionary is the concept of a \emph{spectral triple} (or \emph{unbounded Fredholm module}), introduced by Connes as a non-commutative analogue of the standard Dirac operator on a compact spin manifold. Formally, a spectral triple $(\mathcal{A},\mathcal{H},D)$ over a separable unital C$^{\ast}$-algebra $A$ consists of a norm-dense unital $\ast$-subalgebra $\mathcal{A}\subseteq A$, a faithful representation of $A$ on a Hilbert space $\mathcal{H}$ by bounded operators, and a densely defined, self-adjoint operator $D$ on $\mathcal{H}$, where the operator $D$ is required to have compact resolvent, and the commutators $[D,x]$ must for all $x\in\mathcal{A}$ extend to bounded operators on $\mathcal{H}$.

Beyond capturing differential structures, spectral triples also provide a natural setting for non-commutative metric geometry. As shown by Connes (see \cite{Connes89}), the operator $D$ induces a pseudo-metric on the state space $\mathcal{S}(A)$ of $A$, defined by 
\begin{equation}
d_{D}(\psi,\psi^{\prime}):=\sup\left\{ |\psi(a)-\psi^{\prime}(a)|\;\middle|\;a\in\mathcal{A}\text{ with }\Vert[D,a]\Vert\leq1\right\} .\label{eq:PseudoMetric}
\end{equation}
In the commutative case, this recovers the classical Monge--Kantorovich metric on the space of probability measures of a compact metric space, and the topology induced by this metric coincides with the weak$^{\ast}$-topology. 

However, in the general non-commutative setting, the pseudo-metric $d_{D}$ does not automatically metrize the weak$^{\ast}$-topology on $\mathcal{S}(A)$. Inspired by this observation, Rieffel introduced the theory of \emph{compact quantum metric spaces} (see \cite{Rieffel99, Rieffel04}). While Rieffel's original framework was developed broadly for order unit spaces, the present article focuses exclusively on metric structures arising from spectral triples. Specifically, given a spectral triple $(\mathcal{A},\mathcal{H},D)$, we define the associated \emph{Lipschitz semi-norm} by $L_{D} \colon \mathcal{A}\ni a\mapsto\Vert[D,a]\Vert$. If the Connes pseudo-metric in (\ref{eq:PseudoMetric}) recovers the weak$^{\ast}$-topology on the state space of $A$, the pair $(A,L_{D})$ is called a \emph{compact quantum metric space}, and the semi-norm $L_{D}$ is referred to as a \emph{Lip-norm}. In this case we also say that the triple $(\mathcal{A},\mathcal{H},D)$ is a \emph{spectral metric space} (or also a \emph{metric spectral triple}).

In \cite{Connes89}, Connes constructed natural examples of spectral triples arising from discrete groups equipped with length functions. To be precise, let $G$ be a countable discrete group endowed with a proper length function $\ell$. We consider the reduced group C$^{\ast}$-algebra $C_{\text{r}}^{\ast}(G)\subseteq\mathcal{B}(\ell^{2}(G))$ and its canonical norm-dense $\ast$-subalgebra, the group algebra $\mathbb{C}[G]$. Furthermore, we define a densely defined, self-adjoint operator $D_{\ell}$ on $\ell^{2}(G)$ via the formula
\[
D_{\ell}\delta_{g}:=\ell(g)\delta_{g}\quad\text{for all }g\in G,
\]
where $(\delta_{g})_{g\in G}$ denotes the standard orthonormal basis of the Hilbert space $\ell^{2}(G)$. The resulting triple $(\mathbb{C}[G],\ell^{2}(G),D_{\ell})$ then forms a non-degenerate spectral triple. The investigation of this specific construction through the lens of non-commutative metric geometry was initiated by Rieffel in \cite{Rieffel02}. He demonstrated that for $G=\mathbb{Z}^{d}$ - equipped with a word-length function or the restriction of a norm from $\mathbb{R}^{d}$ - the associated triple is indeed metric. This result was subsequently generalized by Christ and Rieffel in \cite{ChristRieffel} to arbitrary groups and length functions satisfying a bounded doubling condition; in particular, this extension covers all groups of polynomial growth equipped with their respective word-length functions. Moving beyond the amenable setting, Ozawa and Rieffel \cite{OzawaRieffel} proved an analogous statement for word-hyperbolic groups. Further variants of these results have been obtained for higher-order commutator seminorms on groups with the rapid decay property \cite{AntonescuChristensen04}, for twisted reduced group C$^{*}$-algebras \cite{LongWu17,LongWu21}, for discrete quantum groups with rapid decay \cite{BVZ15}, for non-commutative solenoids \cite{FLLP24}, and for certain Iwahori--Hecke algebras of finite-rank right-angled Coxeter systems \cite{KlissePerovic25}. Related crossed-product constructions were developed and studied in \cite{BMR10,HSWZ13,KK21,AKK25, Klisse26}.

Besides the immediate question for an extension of the results above to larger classes of groups, to this day there is no known example of a group $G$, equipped with a word-length function that does not induce a compact quantum metric space; this is, for instance, explicitly mentioned in \cite{Austad26}. Motivated by this, in the present note we consider the \emph{Lamplighter group} $(\mathbb{Z}/2\mathbb{Z})\wr\mathbb{F}_{d}$ with $d \geq 2$, and prove that for arbitrary finite symmetric generating sets the associated spectral triples are not metric.

\begin{introtheorem} \label{MainTheorem} For every integer $d\geq 2$ let $G:=(\mathbb{Z}/2\mathbb{Z})\wr\mathbb{F}_{d}$ be the Lamplighter group of $\mathbb{F}_{d}$ and let $S$ be a finite symmetric generating set. Then the spectral triple $(\mathbb{C}[G],\ell^{2}(G),D_{\ell})$ is not a spectral metric space, where $\ell$ denotes the word-length function associated with $S$. \end{introtheorem}

The proof of Theorem \ref{MainTheorem} employs the construction of suitable averaging operators $(z_{r})_{r\geq1}\in\mathbb{C}[G]$ of the lamp toggles at points of the sphere $\mathcal{S}_{r}$ of radius $r\geq1$ inside the Cayley graph of $\mathbb{F}_{d}$, which is independent of the choice of the finite symmetric generating set $S$. A localization argument for words in $S$ can then be used to prove that $\Vert[D_{\ell},z_{r}]\Vert$ is uniformly bounded in $r$. On the other hand, the Abelian structure of the lamp subgroup dictates that the set $\{z_{r}\mid r\geq1\}$ is uniformly separated.

Theorem \ref{MainTheorem} naturally raises several questions that point toward interesting directions for future research. For instance, it remains open whether the spectral metric space property of word-length spectral triples associated with reduced group C$^{\ast}$-algebras of finitely generated groups is independent of the choice of the generating set. A similar problem regarding the finite-diameter property was posed in Rieffel's foundational paper \cite{Rieffel02}. Furthermore, it is natural to investigate whether there exist finitely generated amenable or rapid decay groups that fail to yield compact quantum metric spaces.

\vspace{3mm}


\subsection*{Acknowledgements}

The author acknowledges the use of GPT-5.6 Sol as an exploratory tool to assist in finding the counterexample. The AI was used under the author's strict mathematical guidance. All mathematical content and arguments were rigorously reviewed, verified, and substantially revised by the author, who assumes full responsibility for the final manuscript. Finally, the author acknowledges Oliver Niehues for his invaluable perspective and general encouragement.

\vspace{5mm}


\section{Preliminaries and notation}

\vspace{3mm}

\subsection{General notation}

We write $\mathbb{N}:=\{0,1,\ldots\}$ and $\mathbb{N}_{\geq 1}:=\{1,2,\ldots\}$. For a statement \(P\), we set \(\delta(P)=1\) if \(P\) is true and \(\delta(P)=0\) otherwise. Given a set \(E\), we denote its cardinality by \(\#E\), its power set by \(\mathcal{P}(E)\), and its characteristic function by \(\chi_E\).

For operators \(x\) and \(y\), their commutator is denoted by $[x,y]:=xy-yx$. The identity element of a group is always denoted by \(e\). Finally, for a graph \(\Gamma\), we denote its vertex and edge sets by \(V\Gamma\) and \(E\Gamma\), respectively.

\vspace{3mm}


\subsection{Spectral triples and compact quantum metric spaces}

Spectral triples, introduced by Connes in \cite{Connes89}, constitute a fundamental concept within the framework of non-commutative geometry.

\begin{definition} Let $A$ be a separable unital C$^{\ast}$-algebra. A\emph{ spectral triple} $(\mathcal{A},\mathcal{H},D)$ on $A$ consists of a $\ast$-representation $\pi:A\rightarrow\mathcal{B}(\mathcal{H})$, a norm dense unital $\ast$-subalgebra $\mathcal{A}$ of $A$ and a densely defined self-adjoint operator $D$ on $\mathcal{H}$ such that $(1+D^{2})^{-\frac{1}{2}}$ is compact, and such that for every $a\in\mathcal{A}$, the domain of $D$ is invariant under $\pi(a)$ and the commutator $[D,\pi(a)]$ is bounded. The operator $D$ is referred to as the triple's \emph{Dirac operator}. \end{definition}

Following Connes \cite{Connes89}, any spectral triple $(\mathcal{A},\mathcal{H},D)$ endows $\mathcal{A}$ with a \emph{Lipschitz semi-norm} $L_{D}$, defined by $L_{D}(a):=\left\Vert [D,\pi(a)]\right\Vert$. Formally, $L_{D}:\mathcal{A}\rightarrow\mathbb{R}_{+}$ is a semi-norm whose domain is a dense subspace of $A$ containing the identity $1$, and satisfying $L_{D}(1)=0$. It induces a pseudo-metric $d_{D}:\mathcal{S}(A)\times\mathcal{S}(A)\rightarrow[0,\infty]$ on the state space $\mathcal{S}(A)$ of $A$, given by 
\[
d_{D}(\psi,\psi^{\prime}):=\sup\left\{ |\psi(a)-\psi^{\prime}(a)|\;\middle|\;a\in\mathcal{A}\text{ with }\Vert[D,a]\Vert\leq1\right\} .
\]
Observe that $d_{D}$ may assume the value $+\infty$.

A fundamental problem within this framework is to determine when the metric topology induced by $d_{D}$ on $\mathcal{S}(A)$ coincides with the weak$^{\ast}$-topology (see \cite{Rieffel98,Rieffel99}). This is the defining property of a compact quantum metric space. A necessary condition for this to hold is that the spectral triple $(\mathcal{A},\mathcal{H},D)$ is \emph{non-degenerate} in the sense that the representation $\pi$ is faithful, and the commutator $[D,\pi(a)]$ must vanish if and only if $a\in\mathbb{C}1$. Whenever the representation is faithful, it is customary to suppress $\pi$ in the notation, viewing both $\mathcal{A}$ and $A$ as $\ast$-subalgebras of $\mathcal{B}(\mathcal{H})$.

\begin{definition}[{\cite[Definition 5.1]{Rieffel99} and \cite[Definition 2.2]{Rieffel04}}] Let $(\mathcal{A},\mathcal{H},D)$ be a non-degenerate spectral triple, and define $L_{D}$ and $d_{D}$ as above. If the metric topology induced by $d_{D}$ on the state space $\mathcal{S}(A)$ coincides with the weak$^{\ast}$-topology, then $L_{D}$ is called a \emph{Lip-norm}. In this case, the pair $(A,L_{D})$ is called a \emph{compact quantum metric space}, and the triple $(\mathcal{A},\mathcal{H},D)$ is referred to as a \emph{spectral metric space} (or a \emph{metric spectral triple}). \end{definition}

Building on \cite{Rieffel98}, the following convenient characterization was established by Ozawa and Rieffel.

\begin{proposition}[{\cite[Proposition 1.3]{OzawaRieffel}}] \label{Characterization} Let $(\mathcal{A},\mathcal{H},D)$ be a non-degenerate spectral triple over a C$^{\ast}$-algebra $A$, and let $L_{D}$ be defined as above. Then the pair $(A,L_{D})$ defines a compact quantum metric space if and only if for every state $\phi$ on $A$ the set
\[
\{a\in\mathcal{A}\mid L_{D}(a)\leq1\text{ and }\phi(a)=0\}
\]
is totally bounded in $A$. \end{proposition}

\vspace{3mm}


\section{Word-lengths on $(\mathbb{Z}/2\mathbb{Z})\wr\mathbb{F}_{d}$}
\vspace{3mm}

Consider the Lamplighter group
\begin{eqnarray*}
G:=(\mathbb{Z}/2\mathbb{Z})\wr\mathbb{F}_{d}=\left(\bigoplus_{\mathbb{F}_{d}}\mathbb{Z}/2\mathbb{Z}\right)\rtimes\mathbb{F}_{d} \quad \text{ with } d \geq 2,
\end{eqnarray*}
where $\mathbb{F}_{d}$ is the free group of rank $d \geq 2$, freely generated by elements $a_1, \ldots, a_d$. Furthermore, let $\mathcal{F}(\mathbb{F}_{d})$ be the set of finite subsets of $\mathbb{F}_{d}$. For notational convenience, we usually express an element $\mathbf{x}=(x_{g})_{g\in\mathbb{F}_{d}}\in\bigoplus_{\mathbb{F}_{d}}\mathbb{Z}/2\mathbb{Z}$ as $\mathbf{x}=\chi_{F}$ with $F:=\{g\in\mathbb{F}_{d}\mid x_{g}=1\}\in\mathcal{F}(\mathbb{F}_{d})$, so that
\[
(\chi_{E},g)(\chi_{F},h)=(\chi_{E\bigtriangleup gF},gh)\quad\text{ for }E,F\in\mathcal{F}(\mathbb{F}_{d}),g,h\in\mathbb{F}_{d}.
\]
Here, $E\bigtriangleup F:=(E\cup F)\setminus(E\cap F)$ denotes the symmetric difference of the sets $E,F\in\mathcal{F}(\mathbb{F}_{d})$.

Given a finite symmetric generating set $S$ of $G$, we write $\ell_{S}$ for the associated word-length function on $G$.

Let $\mathcal{T}:=\text{Cay}(\mathbb{F}_{d},S_{0})$ be the (undirected) Cayley graph of $\mathbb{F}_{d}$ with respect to the symmetric generating set $S_{0}:=\{a_1,{a_1}^{-1},\ldots, a_d,{a_d}^{-1}\}$. Then $\mathcal{T}$ is a $2d$-regular tree, and we denote its vertex set and its edge set by $V\mathcal{T}=\mathbb{F}_{d}$ and $E\mathcal{T}$, respectively. We view $\mathcal{T}$ as a discrete metric space with the graph metric $d_{\mathcal{T}}$. Observe that $\#\mathcal{S}_{r}=2d(2d-1)^{r-1}$ for every $r\geq1$, where 
\[
\mathcal{S}_{r}:=\left\{ g\in\mathbb{F}_{d}\mid d_{\mathcal{T}}(e,g)=r\right\} .
\]

For $E\in\mathcal{F}(\mathbb{F}_{d})$ and $g\in\mathbb{F}_{d}$, we write $\mathcal{H}(E,g)$ for the smallest subtree of $\mathcal{T}$ containing the set $E\cup\{e,g\}$. For elements $g,h\in\mathbb{F}_{d}$, we write $[g,h]_{\mathcal{T}}\subseteq V\mathcal{T}$ for the unique geodesic segment in $\mathcal{T}$ connecting $g$ and $h$. Finally,
\[
N_{R}(E):=\left\{ x\in V\mathcal{T}\mid d_{\mathcal{T}}(x,E) \leq R\right\} \quad \text{ for }E\subseteq\mathbb{F}_{d},R\geq0,
\]
where $d_{\mathcal{T}}(x,E):=\inf\{d_{\mathcal{T}}(x,y)\mid y\in E\}$.

\begin{lemma} \label{ContainmentLemma} Given a finite symmetric generating set $S$ of $G$, express each generator $s\in S$ uniquely in the form $s=(\chi_{E_{s}},g_{s})$ with $E_{s}\in\mathcal{F}(\mathbb{F}_{d})$, $g_{s}\in\mathbb{F}_{d}$. Then the constant 
\[
R_{S}:=\max\left\{ d_{\mathcal{T}}(e,g)\mid s\in S,g\in E_{s}\right\} 
\]
 has the following property: Let $E\in\mathcal{F}(\mathbb{F}_{d})$, $g\in\mathbb{F}_{d}$, and let $(\chi_{E_{1}},g_{1}),\ldots,(\chi_{E_{n}},g_{n})\in S$ be elements with $(\chi_{E},g)=(\chi_{E_{1}},g_{1})\cdots(\chi_{E_{n}},g_{n})$. Then the subset
\[
P(E,g):=\bigcup_{i=1}^{n}[g_{1}\cdots g_{i-1},g_{1}\cdots g_{i}]_{\mathcal{T}}\subseteq V\mathcal{T}
\]
satisfies $\mathcal{H}(E,g)\subseteq N_{R_{S}}(P(E,g))$. \end{lemma}

\begin{proof}  By
\[
(\chi_{E},g)=(\chi_{E_{1}},g_{1})\cdots(\chi_{E_{n}},g_{n})=(\chi_{E_{1}\bigtriangleup g_{1}E_{2}\bigtriangleup\cdots\bigtriangleup(g_{1}\cdots g_{n-1})E_{n}},g_{1}\cdots g_{n}),
\]
an element $h\in E$ must occur in an odd number of the sets $E_{1},g_{1}E_{2},\ldots,(g_{1}\cdots g_{n-1})E_{n}$, so that in particular $h\in(g_{1}\cdots g_{i-1})E_{i}$ for at least one $1\leq i\leq n$. But then $d_{\mathcal{T}}(h,g_{1}\cdots g_{i-1})\leq R_{S}$, and hence $E\subseteq N_{R_{S}}(P(E,g))$.

The closed $R_{S}$-neighbourhood of a connected subtree of a tree is again a connected subtree, and is therefore geodesically convex. It follows that $N_{R_{S}}(P(E,g))$ is geodesically convex. Furthermore, $N_{R_{S}}(P(E,g))$ contains both $E$ and the vertices $e$ and $g$. By definition, $\mathcal{H}(E,g)$ is the smallest subtree of $\mathcal{T}$ containing the set $E\cup\{e,g\}$; it follows that $\mathcal{H}(E,g)\subseteq N_{R_{S}}(P(E,g))$, as claimed. \end{proof}

\begin{remark} Observe that the definition of $P(E,g)$ in Lemma \ref{ContainmentLemma} depends on the choice of the factorization of the element $(\chi_{E},g)\in G$ in terms of the generating set $S$; for notational convenience we suppress this dependence in the notation. \end{remark}

The following proposition permits working with an arbitrary finite generating set $S$ in the later arguments. Note that it does not compare $\ell_{S}$ with another word-length at the level of commutator seminorms, but uses the finite range of the lamp configurations and base displacements occurring in the generators in $S$.

\begin{proposition} \label{WordLengthEstimate} For every finite symmetric generating set $S$ of $G$, there exists a constant $C_{S}>0$ such that
\[
\left|\ell_{S}(\chi_{E\cup\{x\}},g)-\ell_{S}(\chi_{E},g)\right|\leq C_{S}\left(1+d_{\mathcal{T}}(x,\mathcal{H}(E,g))\right)
\]
for every $E\in\mathcal{F}(\mathbb{F}_{d})$, $g\in\mathbb{F}_{d}$, and $x\in\mathbb{F}_{d}\setminus E$. \end{proposition}

\begin{proof} As before, express each generator $s\in S$ uniquely in the form $s=(\chi_{E_{s}},g_{s})$ with $E_{s}\in\mathcal{F}(\mathbb{F}_{d})$, $g_{s}\in\mathbb{F}_{d}$, and set 
\begin{eqnarray*}
J_{S}:&=&\max\left\{ d_{\mathcal{T}}(e,g_{s})\mid s\in S\right\} ,\\
c_{S}:&=& \max\left\{
\ell_{S}(\chi_{\emptyset},a_i),
\ell_{S}(\chi_{\emptyset},a_i^{-1})
\;\middle|\;1\leq i\leq d
\right\},\\
d_{S}:&=&\ell_{S}(\chi_{\{e\}},e).
\end{eqnarray*}
Let $s_{1},\ldots,s_{m}\in S$ be elements with $(\chi_{E},g)=s_{1}\cdots s_{m}$, where $m=\ell_S(\chi_{E},g)$. We may assume that $m\geq1$. Set $r:=d_{\mathcal{T}}(x,\mathcal{H}(E,g))$ and choose $y\in\mathcal{H}(E,g)$ with $d_{\mathcal{T}}(x,y)=r$. By Lemma \ref{ContainmentLemma}, there exists a point $z\in\bigcup_{i=1}^{m}[g_{s_{1}}\cdots g_{s_{i-1}},g_{s_{1}}\cdots g_{s_{i}}]_{\mathcal{T}}$ with 
\[
d_{\mathcal{T}}(y,z)\leq R_{S}=\max\left\{ d_{\mathcal{T}}(e,h)\mid s\in S,h\in E_{s}\right\} .
\]
Since $z$ is contained in one of the segments $[g_{s_{1}}\cdots g_{s_{i-1}},g_{s_{1}}\cdots g_{s_{i}}]_{\mathcal{T}}$, whose length is at most $J_{S}$, we have
\[
d_{\mathcal{T}}(x,g_{s_{1}}\cdots g_{s_{i-1}})\leq d_{\mathcal{T}}(x,y)+d_{\mathcal{T}}(y,z)+d_{\mathcal{T}}(z,g_{s_{1}}\cdots g_{s_{i-1}})\leq r+R_{S}+J_{S}.
\]
Since $x\notin E$, we can express $(\chi_{E\cup\{x\}},g)$ as 
\begin{eqnarray*}
(\chi_{E\cup\{x\}},g) &=& (\chi_{E\bigtriangleup\{x\}},g)=(\chi_{\{x\}},e)(\chi_{E_{s_1}},g_{s_{1}})\cdots(\chi_{E_{s_m}},g_{s_{m}}) \\
&=& \left((\chi_{E_{s_1}},g_{s_{1}})\cdots(\chi_{E_{s_{i-1}}},g_{s_{i-1}})\right)(\chi_{\{(g_{s_{1}}\cdots g_{s_{i-1}})^{-1}x\}},e)\left((\chi_{E_{s_i}},g_{s_{i}})\cdots(\chi_{E_{s_m}},g_{s_{m}})\right),
\end{eqnarray*}
so that
\begin{eqnarray}
\nonumber
\ell_{S}(\chi_{E\cup\{x\}},g)-\ell_{S}(\chi_{E},g) &\leq& \ell_{S}(\chi_{\{(g_{s_{1}}\cdots g_{s_{i-1}})^{-1}x\}},e)\\
\nonumber
&\leq& 2\ell_{S}\left(\chi_{\emptyset},(g_{s_{1}}\cdots g_{s_{i-1}})^{-1}x\right)+\ell_{S}(\chi_{\{e\}},e)\\
\nonumber
&\leq&2c_{S}d_{\mathcal{T}}(g_{s_{1}}\cdots g_{s_{i-1}},x)+d_{S}\\
&\leq& 2c_{S}(r+R_{S}+J_{S})+d_{S}.\label{eq:FirstInequality}
\end{eqnarray}

Conversely, let $t_{1},\ldots,t_{n}\in S$ be elements with 
\[
(\chi_{E\cup\{x\}},g)=t_{1}\cdots t_{n}=(\chi_{E_{t_{1}}\bigtriangleup g_{t_{1}}E_{t_{2}}\bigtriangleup\cdots\bigtriangleup(g_{t_{1}}\cdots g_{t_{n-1}})E_{t_{n}}},g_{t_{1}}\cdots g_{t_{n}}),
\]
where $n= \ell_S (\chi_{E\cup\{x\}},g)$. Since $x$ belongs to $E\cup\{x\}$, it occurs in an odd number of the sets $E_{t_{1}},g_{t_{1}}E_{t_{2}},\ldots,(g_{t_{1}}\cdots g_{t_{n-1}})E_{t_{n}}$. In particular, there exists $1\leq j\leq n$ with $x\in(g_{t_{1}}\cdots g_{t_{j-1}})E_{t_{j}}$, so that $(g_{t_{1}}\cdots g_{t_{j-1}})^{-1}x\in E_{t_{j}}$, and hence
\[
d_{\mathcal{T}}(x,g_{t_{1}}\cdots g_{t_{j-1}})\leq\max\left\{ d_{\mathcal{T}}(e,h)\mid s\in S,h\in E_{s}\right\} =R_{S}.
\]
By a similar calculation as before,
\begin{eqnarray*}
(\chi_{E},g) &=& (\chi_{\{x\}},e)(\chi_{E\cup\{x\}},g)=(\chi_{\{x\}},e)(\chi_{E_{t_{1}}},g_{t_{1}})\cdots(\chi_{E_{t_{n}}},g_{t_{n}})\\
&=& \left((\chi_{E_{t_{1}}},g_{t_{1}})\cdots(\chi_{E_{t_{j-1}}},g_{t_{j-1}})\right)(\chi_{\{(g_{t_{1}}\cdots g_{t_{j-1}})^{-1}x\}},e)\left((\chi_{E_{t_{j}}},g_{t_{j}})\cdots(\chi_{E_{t_{n}}},g_{t_{n}})\right),
\end{eqnarray*}
so that 
\begin{eqnarray}
\nonumber
\ell_{S}(\chi_{E},g)-\ell_{S}(\chi_{E\cup\{x\}},g) &\leq& \ell_{S}(\chi_{\{(g_{t_{1}}\cdots g_{t_{j-1}})^{-1}x\}},e) \\
\nonumber
&\leq& 2\ell_{S}\left(\chi_{\emptyset},(g_{t_{1}}\cdots g_{t_{j-1}})^{-1}x\right)+\ell_{S}(\chi_{\{e\}},e) \\
\nonumber
&\leq& 2c_{S}d_{\mathcal{T}}(x, g_{t_{1}}\cdots g_{t_{j-1}})+\ell_{S}(\chi_{\{e\}},e)\\
&\leq& 2c_{S}R_{S}+d_{S}.\label{eq:SecondInequality}
\end{eqnarray}

Combining (\ref{eq:FirstInequality}) and (\ref{eq:SecondInequality}), we obtain with $C_{S}:=2c_{S}(R_{S}+J_{S}+1)+d_{S}$ the desired inequality. \end{proof}

\vspace{3mm}


\section{Word-length spectral triples of $(\mathbb{Z}/2\mathbb{Z})\wr\mathbb{F}_{d}$}

\vspace{3mm}

The purpose of this section is to prove Theorem \ref{MainTheorem}. We keep the assumptions and notation of the previous section; in particular, $G:=(\mathbb{Z}/2\mathbb{Z})\wr\mathbb{F}_{d}$ will denote the Lamplighter group of $\mathbb{F}_{d}$. Let $\lambda\colon G\rightarrow\mathcal{U}(\ell^{2}(G))$ be the associated left-regular representation. As in the introduction, we consider the reduced group C$^{\ast}$-algebra $C_{\text{r}}^{\ast}(G)\subseteq\mathcal{B}(\ell^{2}(G))$ and its canonical norm-dense $\ast$-subalgebra, the group algebra $\mathbb{C}[G]$. Given a finite symmetric generating set $S$ of $G$, we define for $\ell:=\ell_{S}$ a densely defined, self-adjoint operator $D_{\ell}$ on $\ell^{2}(G)$ via $D_{\ell}\delta_{g}:=\ell(g)\delta_{g}$ for all $g\in G$, where $(\delta_{g})_{g\in G}$ denotes the standard orthonormal basis of the Hilbert space $\ell^{2}(G)$. The resulting triple $(\mathbb{C}[G],\ell^{2}(G),D_{\ell})$ then forms a non-degenerate spectral triple.

\vspace{3mm}


\subsection{Elements supported on lamp spheres}

Recall that $\mathcal{S}_{r}:=\left\{ g\in\mathbb{F}_{d}\mid d_\mathcal{T}(e,g) =r\right\} $ with $\#\mathcal{S}_{r}=2d(2d-1)^{r-1}$ for $r\geq1$. Similar to before, for $E\in\mathcal{F}(\mathbb{F}_{d})$ and $g,x\in\mathbb{F}_{d}$, we write $\mathcal{H}_{x}(E,g)$ for the smallest subtree of $\mathcal{T}$ containing the set $(E\setminus\{x\})\cup\{e,g\}$. Furthermore, we denote the set of elements in $\mathbb{F}_{d}$ starting with $x\in\mathbb{F}_{d}$ by
\[
\mathcal{T}_{x}:=\left\{ g\in\mathbb{F}_{d}\mid d_{\mathcal{T}}(x,g)=d_{\mathcal{T}}(g,e)-d_{\mathcal{T}}(x,e)\right\} ,
\]
and let $v_{k}(x)\in\mathbb{F}_{d}$ be the unique prefix of $x$ of length $0\leq k\leq d_{\mathcal{T}}(e,x)$.

\begin{lemma} \label{Identity} For every $r\geq1$, $E\in\mathcal{F}(\mathbb{F}_{d})$, $g\in\mathbb{F}_{d}$, and $x\in\mathcal{S}_{r}$, the following equality holds: 
\[
d_{\mathcal{T}}(x,\mathcal{H}_{x}(E,g))=\sum_{k=1}^{r}\delta\left(g\notin \mathcal{T}_{v_{k}(x)}\right)\delta\left((E\setminus\{x\})\cap \mathcal{T}_{v_{k}(x)}=\emptyset\right).
\]
\end{lemma}

\begin{proof} First note that the set $\mathcal{H}_{x}(E,g)$ can be expressed as
\begin{equation}
\mathcal{H}_{x}(E,g)=[e,g]_{\mathcal{T}}\cup\bigcup_{y\in E\setminus\{x\}}[e,y]_{\mathcal{T}}.\label{eq:Decomposition}
\end{equation}

For a fixed $1\leq k\leq r$, let $\varepsilon_{k}(x)\in E\mathcal{T}$ be the edge joining $v_{k-1}(x)$ and $v_{k}(x)$. For every $y\in\mathbb{F}_{d}$, the geodesic $[e,y]_{\mathcal{T}}$ contains the edge $\varepsilon_{k}(x)$ if and only if $y\in \mathcal{T}_{v_{k}(x)}$. Using (\ref{eq:Decomposition}), it follows that $\varepsilon_{k}(x)$ belongs to $\mathcal{H}_{x}(E,g)$ if and only if at least one point of $\{g\}\cup(E\setminus\{x\})$ belongs to $\mathcal{T}_{v_{k}(x)}$. Equivalently, $\varepsilon_{k}(x)\notin\mathcal{H}_{x}(E,g)$ if and only if $g\notin \mathcal{T}_{v_{k}(x)}$ and $(E\setminus\{x\})\cap \mathcal{T}_{v_{k}(x)}=\emptyset$. We thereby obtain that
\[
\sum_{k=1}^{r}\delta\left(g\notin \mathcal{T}_{v_{k}(x)}\right)\delta\left((E\setminus\{x\})\cap \mathcal{T}_{v_{k}(x)}=\emptyset\right)=\sum_{k=1}^{r}\delta\left(\varepsilon_{k}(x)\notin\mathcal{H}_{x}(E,g)\right).
\]

Since $\mathcal{H}_{x}(E,g)$ is a connected subtree of $\mathcal{T}$ containing $e$, its intersection with the geodesic $[e,x]_{\mathcal{T}}$ is an initial segment of $[e,x]_{\mathcal{T}}$; thus, there exists an integer $0 \leq j\leq r$ with $\mathcal{H}_{x}(E,g)\cap[e,x]_{\mathcal{T}}=[e,v_{j}(x)]_{\mathcal{T}}$. The vertex $v_{j}(x)$ is the closest point of $\mathcal{H}_{x}(E,g)$ to $x$, so that 
\[
d_{\mathcal{T}}(x,\mathcal{H}_{x}(E,g))=d_{\mathcal{T}}(x,v_{j}(x))=r-j.
\]
Since $r-j$ is exactly the number of edges $\varepsilon_{k}$ with $1\leq k\leq r$ that do not belong to $\mathcal{H}_{x}(E,g)$, it follows that $\sum_{k=1}^{r}\delta\left(\varepsilon_{k}(x)\notin\mathcal{H}_{x}(E,g)\right)=d_{\mathcal{T}}(x,\mathcal{H}_{x}(E,g))$, hence finishing the proof. \end{proof}

\begin{lemma} \label{QDefinition} For $r\geq1$, $1\leq k\leq r$, and $x\in\mathcal{S}_{k}$, the map
\[
\ell^{2}(G)\ni\delta_{(\chi_{E},g)}\mapsto\delta\left(g\notin\mathcal{T}_{x}\right)\sum_{y\in\mathcal{S}_{r}\cap\mathcal{T}_{x}}\delta\left((E\setminus\{y\})\cap\mathcal{T}_{x}=\emptyset\right)\delta_{(\chi_{E\bigtriangleup\{y\}},g)}
\]
extends to a bounded linear operator $Q_{x,r}\in\mathcal{B}(\ell^{2}(G))$ with $\Vert Q_{x,r}\Vert\leq(2d-1)^{\frac{r-k}{2}}$. \end{lemma}

\begin{proof} Define the orthogonal subspaces 
\begin{eqnarray*}
\mathcal{H}_{1}:&=&\overline{\text{Span}}\left(\bigcup_{E\in\mathcal{F}(\mathbb{F}_{d}):E\cap\mathcal{T}_{x}=\emptyset}\left\{ \delta_{(\chi_{E},g)}\mid g\in\mathbb{F}_{d}\setminus\mathcal{T}_{x}\right\} \right) \\
\mathcal{H}_{2}:&=&\overline{\text{Span}}\left(\bigcup_{E\in\mathcal{F}(\mathbb{F}_{d}):E\cap\mathcal{T}_{x}=\emptyset}\bigcup_{y\in\mathcal{S}_{r}\cap\mathcal{T}_{x}}\left\{ \delta_{(\chi_{E\cup\{y\}},g)}\mid g\in\mathbb{F}_{d}\setminus\mathcal{T}_{x}\right\} \right)
\end{eqnarray*}
of $\ell^{2}(G)$, and let $P_{1}$ and $P_{2}$ be the corresponding orthogonal projections onto $\mathcal{H}_{1}$ and $\mathcal{H}_{2}$.\\

We distinguish four cases:\\

\begin{itemize}
\item \emph{Case 1}: For $g\in\mathcal{T}_{x}$ and $E\in\mathcal{F}(\mathbb{F}_{d})$, one has that $Q_{x,r}\delta_{(\chi_{E},g)}=0$. 
\item \emph{Case 2}: For $g\in\mathbb{F}_{d}\setminus\mathcal{T}_{x}$ and $E\in\mathcal{F}(\mathbb{F}_{d})$ with $E\cap\mathcal{T}_{x}=\emptyset$, the equality
\[
Q_{x,r}\delta_{(\chi_{E},g)}=\sum_{y\in\mathcal{S}_{r}\cap\mathcal{T}_{x}}\delta_{(\chi_{E\cup\{y\}},g)}
\]
holds.
\item \emph{Case 3}: For $g\in\mathbb{F}_{d}\setminus\mathcal{T}_{x}$ and $E\in\mathcal{F}(\mathbb{F}_{d})$ with $E\cap\mathcal{T}_{x}=\{y\}$ for some $y$,
\[
Q_{x,r}\delta_{(\chi_{E},g)}=\begin{cases}
\delta_{(\chi_{E\setminus\{y\}},g)}, & \text{if }y\in\mathcal{S}_{r},\\
0, & \text{if }y\notin\mathcal{S}_{r}.
\end{cases}
\]
\item \emph{Case 4}: For $g\in\mathbb{F}_{d}\setminus\mathcal{T}_{x}$ and $E\in\mathcal{F}(\mathbb{F}_{d})$ with $\#(E\cap\mathcal{T}_{x})\geq2$, $Q_{x,r}\delta_{(\chi_{E},g)}=0$.
\end{itemize}

The calculations above imply that $Q_{x,r}$ vanishes on $(\mathcal{H}_{1}\oplus\mathcal{H}_{2})^{\perp}$ and that $Q_{x,r}\mathcal{H}_{1}\subseteq\mathcal{H}_{2}$ and $Q_{x,r}\mathcal{H}_{2}\subseteq\mathcal{H}_{1}$. It follows that $Q_{x,r}=P_{1}Q_{x,r}P_{2}+P_{2}Q_{x,r}P_{1}$ and therefore
\[
\Vert Q_{x,r}\Vert=\max\{\Vert P_{1}Q_{x,r}P_{2}\Vert,\Vert P_{2}Q_{x,r}P_{1}\Vert\}.
\]
For every finite sum of the form $\xi:=\sum_{E\in\mathcal{F}(\mathbb{F}_{d}):E\cap\mathcal{T}_{x}=\emptyset}\sum_{g\in\mathbb{F}_{d}\setminus\mathcal{T}_{x}}\xi_{E,g}\delta_{(\chi_{E},g)}\in\mathcal{H}_{1}$ with complex-valued coefficients,
\begin{eqnarray*}
\left\Vert Q_{x,r}\xi\right\Vert ^{2} &=& \left\Vert \sum_{E\in\mathcal{F}(\mathbb{F}_{d}):E\cap\mathcal{T}_{x}=\emptyset}\sum_{g\in\mathbb{F}_{d}\setminus\mathcal{T}_{x}}\sum_{y\in\mathcal{S}_{r}\cap\mathcal{T}_{x}}\xi_{E,g}\delta_{(\chi_{E\cup\{y\}},g)}\right\Vert ^{2} \\
&=& \#(\mathcal{S}_{r}\cap\mathcal{T}_{x})\sum_{E\in\mathcal{F}(\mathbb{F}_{d}):E\cap\mathcal{T}_{x}=\emptyset}\sum_{g\in\mathbb{F}_{d}\setminus\mathcal{T}_{x}}|\xi_{E,g}|^{2} \\
&=& \#(\mathcal{S}_{r}\cap\mathcal{T}_{x})\Vert\xi\Vert^{2},
\end{eqnarray*}
while for every finite sum $\eta:=\sum_{E\in\mathcal{F}(\mathbb{F}_{d}):E\cap\mathcal{T}_{x}=\emptyset}\sum_{y\in\mathcal{S}_{r}\cap\mathcal{T}_{x}}\sum_{g\in\mathbb{F}_{d}\setminus\mathcal{T}_{x}}\eta_{E,g,y}\delta_{(\chi_{E\cup\{y\}},g)}\in\mathcal{H}_{2}$ with complex-valued coefficients,
\begin{eqnarray*}
\left\Vert Q_{x,r}\eta\right\Vert ^{2} &=& \left\Vert \sum_{E\in\mathcal{F}(\mathbb{F}_{d}):E\cap\mathcal{T}_{x}=\emptyset}\sum_{y\in\mathcal{S}_{r}\cap\mathcal{T}_{x}}\sum_{g\in\mathbb{F}_{d}\setminus\mathcal{T}_{x}}\eta_{E,g,y}\delta_{(\chi_{E},g)}\right\Vert ^{2} \\
&=& \sum_{E\in\mathcal{F}(\mathbb{F}_{d}):E\cap\mathcal{T}_{x}=\emptyset}\sum_{g\in\mathbb{F}_{d}\setminus\mathcal{T}_{x}}\left|\sum_{y\in\mathcal{S}_{r}\cap\mathcal{T}_{x}}\eta_{E,g,y}\right|^{2} \\
&\leq& \#(\mathcal{S}_{r}\cap\mathcal{T}_{x})\sum_{E\in\mathcal{F}(\mathbb{F}_{d}):E\cap\mathcal{T}_{x}=\emptyset}\sum_{g\in\mathbb{F}_{d}\setminus\mathcal{T}_{x}}\sum_{y\in\mathcal{S}_{r}\cap\mathcal{T}_{x}}\left|\eta_{E,g,y}\right|^{2} \\
&=& \#(\mathcal{S}_{r}\cap\mathcal{T}_{x})\Vert\eta\Vert^{2}.
\end{eqnarray*}
Thereby, 
\[
\max\{\Vert P_{1}Q_{x,r}P_{2}\Vert,\Vert P_{2}Q_{x,r}P_{1}\Vert\} \leq \sqrt{\#(\mathcal{S}_{r}\cap\mathcal{T}_{x})}=(2d-1)^{\frac{r-k}{2}}.
\]
 This implies the claim. \end{proof}

\begin{proposition} \label{BrDefinition} For $r\geq1$, the map
\begin{equation}
\ell^{2}(G)\ni\delta_{(\chi_{E},g)}\mapsto\frac{1}{\#\mathcal{S}_{r}}\sum_{x\in\mathcal{S}_{r}}\left(1+d_{\mathcal{T}}(x,\mathcal{H}_{x}(E,g))\right)\delta_{(\chi_{E}\bigtriangleup\{x\},g)}\label{eq:MapDefinition}
\end{equation}
extends to a bounded linear operator $B_{r}\in\mathcal{B}(\ell^{2}(G))$ with 
\[
\Vert B_{r}\Vert\leq1+\frac{1}{1-(2d-1)^{-\frac{1}{2}}}.
\]
\end{proposition}

\begin{proof} By the identity in Lemma \ref{Identity} and the definition of $Q_{x,r}$, one has
\[
B_{r}=\frac{1}{\#\mathcal{S}_{r}}\sum_{x\in\mathcal{S}_{r}}\lambda_{(\chi_{\{x\}},e)}+\frac{1}{\#\mathcal{S}_{r}}\sum_{k=1}^{r}\sum_{x\in\mathcal{S}_{k}}Q_{x,r},
\]
so that the map in (\ref{eq:MapDefinition}) indeed extends to a bounded linear operator. Furthermore, by Lemma \ref{QDefinition},
\begin{eqnarray*}
\Vert B_{r}\Vert &\leq& 1+\frac{1}{\#\mathcal{S}_{r}}\sum_{k=1}^{r}\sum_{x\in\mathcal{S}_{k}}\Vert Q_{x,r}\Vert\leq1+\frac{1}{\#\mathcal{S}_{r}}\sum_{k=1}^{r}(\#\mathcal{S}_{k})(2d-1)^{\frac{r-k}{2}} \\
&=& 1+\sum_{k=1}^{r}(2d-1)^{\frac{k-r}{2}} \leq 1+\frac{1}{1-(2d-1)^{-\frac{1}{2}}}.
\end{eqnarray*}
\end{proof}

\begin{theorem} \label{AveragingOperators} For every finite symmetric generating set $S$ of $G$ and every $r\geq1$, the element $z_{r}:=(\#\mathcal{S}_{r})^{-1}\sum_{g\in\mathcal{S}_{r}}\lambda_{(\chi_{\{g\}},e)}\in\mathbb{C}[G]\subseteq C_{\text{r}}^{\ast}(G)$ satisfies for $\ell:=\ell_{S}$ the inequality 
\[
\left\Vert [D_{\ell},z_{r}]\right\Vert \leq C_{S}\left(1+\frac{1}{1-(2d-1)^{-\frac{1}{2}}}\right),
\]
where $C_{S}>0$ is the constant from Proposition \ref{WordLengthEstimate}. \end{theorem}

\begin{proof} Given a finitely supported function $\xi\in\ell^{2}(G)$, write $|\xi|:=\sum_{E\in\mathcal{F}(\mathbb{F}_{d})}\sum_{g\in\mathbb{F}_{d}}|\xi(E,g)|\delta_{(\chi_{E},g)}$. We have that 
\[
\left([D_{\ell},z_{r}]\xi\right)(\chi_{E},g)=\frac{1}{\#\mathcal{S}_{r}}\sum_{x\in\mathcal{S}_{r}}\left(\ell_{S}(\chi_{E},g)-\ell_{S}(\chi_{E\bigtriangleup\{x\}},g)\right)\xi(\chi_{E\bigtriangleup\{x\}},g)
\]
for every $E\in\mathcal{F}(\mathbb{F}_{d})$ and $g\in\mathbb{F}_{d}$, so that by Proposition \ref{WordLengthEstimate}, the definition of $B_{r}$ in Proposition \ref{BrDefinition}, and the equality $\mathcal{H}_{x}(E\bigtriangleup\{x\},g)=\mathcal{H}_{x}(E,g)$ for $x\in\mathcal{S}_{r}$, 
\[
\left|\left([D_{\ell},z_{r}]\xi\right)(\chi_{E},g)\right|\leq\frac{C_{S}}{\#\mathcal{S}_{r}}\sum_{x\in\mathcal{S}_{r}}\left(1+d_{\mathcal{T}}(x,\mathcal{H}_{x}(E,g))\right)\left|\xi(\chi_{E\bigtriangleup\{x\}},g)\right|=C_{S}\left(B_{r}|\xi|\right)(\chi_{E},g).
\]
With $\Vert B_{r}\Vert\leq1+(1-(2d-1)^{-\frac{1}{2}})^{-1}$, one obtains that 
\[
\left\Vert [D_{\ell},z_{r}]\xi\right\Vert  \leq C_{S}\left\Vert B_{r}|\xi|\right\Vert \leq C_{S}\left(1+\frac{1}{1-(2d-1)^{-\frac{1}{2}}}\right)\left\Vert \xi\right\Vert ,
\]
and hence 
\[
\left\Vert [D_{\ell},z_{r}]\right\Vert \leq C_{S}\left(1+\frac{1}{1-(2d-1)^{-\frac{1}{2}}}\right).
\]
\end{proof}

\vspace{3mm}


\subsection{Proof of the main result}

We have now assembled all the ingredients to prove Theorem \ref{MainTheorem}.

\begin{proof}[{Proof of Theorem \ref{MainTheorem}}] Let $G:=(\mathbb{Z}/2\mathbb{Z})\wr\mathbb{F}_{d}$ be the Lamplighter group of $\mathbb{F}_{d}$ with $d \geq 2$, and let $S$ be a finite symmetric generating set. By the characterization in Proposition \ref{Characterization}, it suffices to show that the set 
\[
\mathcal{L}_{1}:=\{a\in\mathbb{C}[G]\mid\tau(a)=0\text{ and } \Vert[D_{\ell},a]\Vert\leq1\}
\]
is not totally bounded, where $\tau$ denotes the canonical tracial state on $C_{\text{r}}^{\ast}(G)$. For this, consider the family of elements
\[
z_{r}:=\frac{1}{\#\mathcal{S}_{r}}\sum_{g\in\mathcal{S}_{r}}\lambda_{(\chi_{\{g\}},e)}\in\mathbb{C}[G]\subseteq C_{\text{r}}^{\ast}(G)
\]
from Theorem \ref{AveragingOperators} and set $\widetilde{z}_{r}:=C^{-1}z_{r}$, where $C:=C_{S}(1+(1-(2d-1)^{-\frac{1}{2}})^{-1})$. Then, $\widetilde{z}_{r}\in\mathcal{L}_{1}$ for every $r\geq1$.

Fix $r\neq r^{\prime}$ with $r,r^{\prime}\geq1$. The coefficient group $B:=\bigoplus_{\mathbb{F}_{d}}\mathbb{Z}/2\mathbb{Z}\leq G$ is Abelian, and the map $B\rightarrow\mathbb{T}$, $\mathbf{x}\mapsto\prod_{g\in\mathcal{S}_{r}}(-1)^{x_{g}}$ defines a group homomorphism on $B$. This group homomorphism induces a multiplicative state $\chi$ on the Abelian C$^{\ast}$-algebra $C_{\text{r}}^{\ast}(B)\subseteq C_{\text{r}}^{\ast}(G)$, so that in particular 
\[
\Vert\widetilde{z}_{r}-\widetilde{z}_{r^{\prime}}\Vert=C^{-1}\Vert z_{r}-z_{r^{\prime}}\Vert\geq C^{-1}|\chi(z_{r}-z_{r^{\prime}})|=2 C^{-1}.
\]
Since $\{\widetilde{z}_{r}\mid r\geq1\}\subseteq\mathcal{L}_{1}$ is thereby a uniformly separated subset, $\mathcal{L}_{1}$ cannot be totally bounded. This completes the proof. \end{proof}

We finish this section with the following natural question.

\begin{question}
Let $\Gamma$ be a non-elementary word-hyperbolic group and let $G:=(\mathbb{Z}/2\mathbb{Z})\wr\Gamma$. Is the spectral triple $(\mathbb{C}[G],\ell^{2}(G),D_{\ell})$ with $\ell:=\ell_S$ for every finite symmetric generating set $S$ of $G$ non-metric?
\end{question}

\vspace{3mm}



\begin{thebibliography}{BBI}

\bibitem{AntonescuChristensen04} C. Antonescu, E. Christensen, \emph{Metrics on group C$^{\ast}$-algebras and a non-commutative Arzelà--Ascoli theorem}, J. Funct. Anal. 214 (2004), no. 2, 247--259.

\bibitem{Austad26} A. Austad, \emph{Quantum metrics from length functions on étale groupoids},  arXiv preprint arXiv:2602.20032 (2026).

\bibitem{AKK25} A. Austad, J. Kaad and D. Kyed, \emph{Quantum metrics on crossed products with groups of polynomial growth}, Trans. Amer. Math. Soc. 378 (2025), no. 3, 1939--1973.

\bibitem{BMR10} J. Bellissard, M. Marcolli, K. Reihani, \emph{Dynamical systems on spectral metric spaces}, arXiv preprint arXiv:1008.4617 (2010).

\bibitem{BVZ15} J. Bhowmick, C. Voigt, J. Zacharias, \emph{Compact quantum metric spaces from quantum groups of rapid decay}, J. Noncommut. Geom. 9 (2015), no. 4, 1175--1200.

\bibitem{ChristRieffel} M. Christ, M. A. Rieffel, \emph{Nilpotent group C$^{\ast}$-algebras as compact quantum metric spaces}, Canad. Math. Bull. 60 (2017), no. 1, 77--94.

\bibitem{Connes89} A. Connes, \emph{Compact metric spaces, Fredholm modules, and hyperfiniteness,} Ergodic Theory Dynam. Systems 9 (1989), no. 2, 207--220.

\bibitem{HSWZ13} A. Hawkins, A. Skalski, S. White, J. Zacharias, \emph{On spectral triples on crossed products arising from equicontinuous actions}, Math. Scand. 113 (2013), no. 2, 262--291.

\bibitem{FLLP24} C. Farsi, T. Landry, N. S. Larsen, J. Packer, \emph{Spectral triples for noncommutative solenoids and a Wiener's lemma}, J. Noncommut. Geom. 18 (2024), no. 4, 1415--1452.

\bibitem{KK21} J. Kaad, D. Kyed, \emph{Dynamics of compact quantum metric spaces}, Ergodic Theory Dynam. Systems 41 (2021), no. 7, 2069--2109.

\bibitem{Klisse26} M. Klisse, \emph{Crossed products as compact quantum metric spaces}, Canad. J. Math. 78 (2026), no. 1, 245--275.

\bibitem{KlissePerovic25} M. Klisse, H. Perovi\'{c}, \emph{Quantum metric structures on Iwahori-Hecke algebras}, arXiv preprint arXiv:2508.07857 (2025).

\bibitem{LongWu17} B. Long, W. Wu, \emph{Twisted group C$^{\ast}$-algebras as compact quantum metric spaces}, Results Math. 71 (2017), no. 3, 911--931.

\bibitem{LongWu21} B. Long, W. Wu, \emph{Twisted bounded-dilation group C$^{\ast}$-algebras as C$^{\ast}$-metric algebras}, Sci. China Math. 64 (2021), no. 3, 547--572.

\bibitem{OzawaRieffel} N. Ozawa, M. Rieffel, \emph{Hyperbolic group C$^{\ast}$-algebras and free-product C$^{\ast}$-algebras as compact quantum metric spaces}, Canad. J. Math. 57 (2005), no. 5, 1056--1079.

\bibitem{Rieffel98} M. Rieffel, \emph{Metrics on states from actions of compact groups}, Doc. Math. 3 (1998), 215--229.

\bibitem{Rieffel99} M. Rieffel, \emph{Metrics on state spaces}, Doc. Math. 4 (1999), 559--600.

\bibitem{Rieffel02} M. Rieffel, \emph{Group C$^{\ast}$-algebras as compact quantum metric spaces}, Doc. Math. 7 (2002), 605--651.

\bibitem{Rieffel04} M. Rieffel, \emph{Gromov--Hausdorff distance for quantum metric spaces}, Mem. Amer. Math. Soc. 168 (2004), no. 796, 1--65.

\vspace{3mm}
 
\end{thebibliography}
\end{document}